\documentclass[11pt,letterpaper]{article}

\usepackage{tk}
\usepackage{hyperref}

\hypersetup{
  linkcolor  = cblue,
  citecolor  = cgreen,
  urlcolor   = corange,
  colorlinks = true,
}

\newcommand{\Wig}{\mathsf{Wig}}
\newcommand{\Mix}{\mathsf{Mix}}
\newcommand{\heavy}{\mathsf{heavy}}
\newcommand{\ppmod}[1]{\,\, (\mathsf{mod}\ #1)}

\title{Hypothesis testing with low-degree polynomials in the Morris class of exponential families}
\author{Dmitriy Kunisky\thanks{Email: \textit{kunisky@cims.nyu.edu}. Partially supported by NSF grants DMS-1712730 and DMS-1719545.}}
\affil{Department of Mathematics, Courant Institute of Mathematical Sciences, New York University}
\date{November 6, 2020}

\begin{document}

\pagenumbering{gobble}

\maketitle

\begin{abstract}
    Analysis of low-degree polynomial algorithms is a powerful, newly-popular method for predicting computational thresholds in hypothesis testing problems.
    One limitation of current techniques for this analysis is their restriction to Bernoulli and Gaussian distributions.
    We expand this range of possibilities by performing the low-degree analysis of hypothesis testing for the Morris class of natural exponential families with quadratic variance function, giving a unified treatment of Gaussian, Poisson, gamma (including exponential and chi-squared), binomial (including Bernoulli), negative binomial (including geometric), and generalized hyperbolic secant distributions.
    We then give several algorithmic applications.
    \begin{enumerate}
    \item In models where a random signal is observed through coordinatewise-independent noise applied in an exponential family, the success or failure of low-degree polynomials is governed by the \emph{$z$-score overlap}, the inner product of $z$-score vectors with respect to the null distribution of two independent copies of the signal.
    \item In the same models, testing with low-degree polynomials exhibits \emph{channel monotonicity}: the above distributions admit a total ordering by computational cost of hypothesis testing, according to a scalar parameter describing how the variance depends on the mean in an exponential family.
    \item In a spiked matrix model with a particular non-Gaussian noise distribution, the low-degree prediction is incorrect unless polynomials with arbitrarily large degree in individual matrix entries are permitted.
        This shows that polynomials summing over self-avoiding walks and variants thereof, as proposed recently by Ding, Hopkins, and Steurer~(2020) for spiked matrix models with heavy-tailed noise, are strictly suboptimal for this model.
        Thus low-degree polynomials appear to offer a tradeoff between robustness and strong performance fine-tuned to specific models.
        Inspired by this, we suggest that a class of problems requiring \emph{exploration before inference}, where an algorithm must first examine the input and then use some intermediate computation to choose a suitable inference subroutine, appears especially difficult for low-degree polynomials.
    \end{enumerate}
\end{abstract}

    % We show that polynomials of a matrix variable where each entry can only be raised to a bounded power fail to distinguish a non-Gaussian spiked matrix model from a null model, even when PCA after a transcendental entrywise transformation succeeds.
    % In particular, these include polynomials summing over self-avoiding walks, as proposed recently by \cite{DHS-2020-SpikedMatrixHeavyTailed} for such problems under heavy-tailed noise.
    % This casts some doubt on the possibility of a single optimal distribution-agnostic testing algorithm for spiked matrix models.
    % If such an algorithm does exist, our result suggests that it cannot be expressed as a low-degree polynomial, which would pose a challenge to a recent research program supposing that most algorithms of theoretical relevance can be expressed in this way.
    % Our result follows from a more general analysis, perhaps of independent interest, of the low-degree likelihood ratio of natural exponential families with quadratic variance function.

\newpage

\tableofcontents

\newpage

\pagenumbering{arabic}

\section{Introduction}

% Other citations: Bresler
% TODO: Gaussian = Poisson => pass from dense to sparse

A powerful framework has emerged recently for making predictions of and producing evidence for computational thresholds in high-dimensional average-case algorithmic problems, which analyzes the performance of algorithms that compute low-degree polynomials.
Originating in results on sum-of-squares optimization \cite{BHKKMP-2019-PlantedClique, HKPRSS-2017-SOSSpectral,HS-2017-BayesianEstimation,Hopkins-2018-Thesis}, this idea has since been fruitfully applied to a wide range of problems \cite{BKW-2019-ConstrainedPCA, DKWB-2019-SubexponentialTimeSparsePCA,BB-2020-ReducibilityStatCompGaps,GJW-2020-LowDegreeOptimization,SW-2020-LowDegreeEstimation,BBKMW-2020-SpectralPlantingColoring,Wein-2020-LowDegreeIndependentSet}.
Perhaps the main attraction of this so-called \emph{low-degree method} is the simplicity of low-degree polynomial algorithms, which are both quite natural and often easier to study than other putatively-optimal families of algorithms like convex relaxations and message-passing techniques.

Before applying the low-degree method in a given situation, we must decide how exactly to make sense of ``a low-degree polynomial succeeding'' at a given computational task.
Besides just producing a sensible definition for a given task---different definitions have been used in the literature for hypothesis testing \cite{Hopkins-2018-Thesis,KWB-2019-NotesLowDegree}, statistical estimation \cite{HS-2017-BayesianEstimation,DHS-2020-SpikedMatrixHeavyTailed,SW-2020-LowDegreeEstimation}, and optimization \cite{GJW-2020-LowDegreeOptimization,Wein-2020-LowDegreeIndependentSet}---practitioners also sometimes make choices in this step to keep the low-degree method analytically tractable.
On the other hand, the problems studied to date with the low-degree method involve only a few different families of probability distributions over their inputs, and often the low-degree analysis takes advantage of an intimate coupling between simplifications made in the success criteria and convenient distribution-specific identities.
Our goal will be to investigate the most common measures of efficacy for low-degree polynomial algorithms for a wider range of probability distributions, probing whether some of the success of the low-degree method might be due to this kind of fortuitous coincidence.

We focus on hypothesis testing: given two sequences of probability distributions $\PP_n$ and $\QQ_n$, we consider whether a low-degree polynomial can correctly distinguish $\by \sim \PP_n$ from $\by \sim \QQ_n$ with high probability as $n \to \infty$.
The most direct way to formalize this question, by analogy with the classical Neyman-Pearson lemma \cite{NP-1933-MostEfficientTests}, is to ask, for some degree bound $D(n)$ corresponding to a computational budget: \emph{do there exist polynomials $f_n \in \RR[\by]$ with $\deg(f(n)) \leq D(n)$ and thresholds $\xi_n \in \RR$ such that}
\begin{equation}
    \label{eq:intro-poly-thresh}
    \lim_{n \to \infty} \PP_n[f_n(\by) > \xi_n] = \lim_{n \to \infty} \QQ_n[f_n(\by) < \xi_n] = 1 \text{ ?}
\end{equation}
Unfortunately, it appears difficult to prove lower bounds against this class of algorithms.
Instead, recent research has focused on the following ``averaged'' version of the above: \emph{do there exist polynomials $f_n \in \RR[\by] \cap L^2(\QQ_n)$ with $\deg(f(n)) \leq D(n)$ such that}
\begin{equation}
    \label{eq:intro-low-deg}
    \lim_{n \to \infty} \, \Ex_{\by \sim \PP_n} f_n(\by) = +\infty \, , \text{ while } \Ex_{\by \sim \QQ_n} f_n(\by)^2 = 1 \text{ ?}
\end{equation}
As we will see in Section~\ref{sec:low-deg}, here the optimal $f_n$ may be computed explicitly using orthogonal polynomials, at least for simple models of hypothesis testing.
The resulting predictions have been remarkably consistent with other, more technically-challenging methods for models including the planted clique problem \cite{BHKKMP-2019-PlantedClique}, the stochastic block model \cite{HS-2017-BayesianEstimation}, and spiked matrix and tensor models \cite{Hopkins-2018-Thesis,KWB-2019-NotesLowDegree,BKW-2019-ConstrainedPCA, DKWB-2019-SubexponentialTimeSparsePCA}.

Yet, our current understanding of this heuristic and the circumstances under which it is accurate remains incomplete in the regard mentioned above: these examples all involve observations $\by$ having only Gaussian or Bernoulli distributions.
Moreover, the low-degree analysis often hinges on special algebraic and analytic properties of these distributions and their orthogonal polynomials.
In this paper, we will describe the results of applying the low-degree method to hypothesis testing in many further exponential families, making new predictions and suggesting some challenging examples that we hope will stimulate further research in this direction.

\subsection{Natural exponential families with quadratic variance function}

We first introduce the distributions we will study, which form \emph{exponential families}.
Throughout this section we follow Morris' presentation in the seminal papers \cite{Morris-1982-NEFQVF,Morris-1983-NEFQVFStats}, which first recognized the many shared statistical properties of these families.
We start by recalling the basic notions of exponential families.
\begin{definition}
    Let $\rho_0$ be a probability measure over $\RR$ which is not a single atom.
    Let $\psi(\theta) \colonequals \log \EE_{x \sim \rho_0}[\exp(\theta x) ]$ and $\Theta \colonequals \{\theta \in \RR: \psi(\theta) < \infty\}$.
    Then, the \emph{natural exponential family (NEF) generated by $\rho_0$} is the family of probability measures $\rho_{\theta}$, for $\theta \in \Theta$, given by
    \begin{equation}
        d\rho_\theta(x) \colonequals \exp(\theta x - \psi(\theta)) d\rho_0(x).
    \end{equation}
\end{definition}

Sometimes, the ``natural parameter'' $\theta$ is the mean of $\rho_{\theta}$ or a translation thereof; however, as the next example shows, the mapping $\theta \mapsto \EE_{x \sim \rho_{\theta}}[x]$ need not be particularly simple in general.
\begin{example}
    Taking $d\rho_0(x) = e^{-x}\One\{x \geq 0\}dx$, we have $\Theta = (-\infty, 1)$, and this generates the NEF of exponential distributions, $d\rho_{\theta}(x) = (1 - \theta)e^{-(1 - \theta)x}\One\{x \geq 0\}dx$.
    The mean of $\rho_{\theta}$ is $\EE_{x \sim \rho_{\theta}}[x] = \frac{1}{1 - \theta}$.
\end{example}
\noindent
Nonetheless,  it is always possible to reparametrize any NEF in terms of the mean in the following way.
The cumulant generating functions of the $\rho_\theta$ are merely translations of $\psi$, $\psi_{\theta}(\eta) \colonequals \log \EE_{x \sim \rho_{\theta}}[\exp(\eta x)] = \psi(\theta + \eta) - \psi(\theta)$.
Therefore, the means and variances of $\rho_{\theta}$ are
\begin{align}
  \mu_\theta &\colonequals \EE_{x \sim \rho_\theta}[x] = \psi_{\theta}^\prime(0) = \psi^\prime(\theta), \\
  \sigma_\theta^2 &\colonequals \EE_{x \sim \rho_\theta}[x^2] - (\EE_{x \sim \rho_\theta} [x])^2 = \psi_{\theta}^{\prime\prime}(0) = \psi^{\prime\prime}(\theta).
\end{align}
Since $\rho_0$ is not an atom, neither is any $\rho_\theta$, and thus $\psi^{\prime\prime}(\theta) = \sigma_\theta^2 > 0$ for all $\theta \in \RR$.
Therefore, $\psi^\prime$ is strictly increasing, and thus one-to-one.
Letting $\Omega \subseteq \RR$ equal the image of $\RR$ under $\psi^\prime$ (some open interval, possibly infinite on either side, of $\RR$), we see that $\rho_\theta$ admits an alternative mean parametrization, as follows.
\begin{definition}
    If $\rho_0$ generates the NEF $\rho_{\theta}$, then we let $\widetilde{\rho}_{\mu} = \rho_{(\psi^\prime)^{-1}(\mu)}$ over $\mu \in \Omega$.
    The \emph{mean-parametrized NEF generated by $\rho_0$} is the family of probability measures $\widetilde{\rho}_{\mu}$, for $\mu \in \Omega$.
\end{definition}

By the same token, within an NEF, the variance is a function of the mean.
In the above setting, we denote this function as follows.
\begin{definition}
    For $\mu \in \Omega$, define the \emph{variance function} $V(\mu) \colonequals \sigma_{(\psi^\prime)^{-1}(\mu)}^2 = \psi^{\prime\prime}((\psi^\prime)^{-1} (\mu))$.
\end{definition}
\noindent
The function $V(\mu)$ is simple for many NEFs that are theoretically important, and its simplicity appears to be a better measure of the ``canonicity'' of an NEF than, e.g., the simplicity of the probability density or mass function.
Specifically, the most important NEFs have $V(\mu)$ a low-degree polynomial: $V(\mu)$ is constant only for the Gaussian NEF with some fixed variance, and linear only for the Poisson NEF and affine transformations thereof.

The situation becomes more interesting for $V(\mu)$ quadratic, which NEFs Morris gave the following name.
\begin{definition}
    If $V(\mu) = v_0 + v_1\mu + v_2\mu^2$ for some $v_i \in \RR$, then we say that $\rho_0$ generates a \emph{natural exponential family with quadratic variance function (NEF-QVF)}.
\end{definition}
\noindent
NEF-QVFs are also sometimes called the \emph{Morris class} of exponential families.
One of the main results of \cite{Morris-1982-NEFQVF} is a complete classification of the NEF-QVFs, as follows.

\begin{table}
    \begin{center}
    \begin{tabular}{lccc}
      \hline
      \textbf{Name} & $d\rho_0(x)$ & \textbf{Support} & $V(\mu)$ \\
      \hline \\[-0.9em]
      Gaussian (variance $\sigma^2 > 0$) & $\frac{1}{\sqrt{2\pi\sigma^2}}\exp(-\frac{1}{2\sigma^2}x^2)dx$ & $\RR$ & $\sigma^2$ \\[0.25em]
      Poisson & $\frac{1}{e}\frac{1}{x!}$ & $\ZZ_{\geq 0}$ & $\mu$ \\[0.25em]
      Gamma (shape $\alpha > 0$) & $\frac{1}{\Gamma(\alpha)}x^{\alpha - 1}e^{-x}dx$ & $(0, +\infty)$ & $\frac{1}{\alpha}\mu^2$ \\[0.25em]
      Binomial ($m$ trials) & $\frac{1}{2^m} \binom{m}{x}$ & $\{0, \ldots, m\}$ & $-\frac{1}{m}\mu^2 + \mu$ \\[0.25em]
      Negative Binomial ($m$ successes) & $\frac{1}{2^{m + x}}\binom{x + m - 1}{x}$ & $\ZZ_{\geq 0}$ & $\frac{1}{m}\mu^2 + \mu$ \\[0.25em]
      Generalized Hyperbolic Secant (shape $r > 0$) & (\cite{Morris-1982-NEFQVF}, Section 5) & $\RR$ & $\frac{1}{r}\mu^2 + r$ \\[0.25em]
      \hline
    \end{tabular}
    \end{center}

    \vspace{-1em}

    \caption{\textbf{The six basic NEF-QVFs.} We describe the six natural exponential families with quadratic variance function from which, according to the results of \cite{Morris-1982-NEFQVF}, any such family can be generated by an affine transformation.
      % Other common distributions occur as special cases: Bernoulli is a special case of binomial, geometric is a special case of negative binomial, and exponential and chi-squared are both special cases of gamma.
      The sixth ``generalized hyperbolic secant'' family is more complicated to describe, but one representative distribution generating the $r = 1$ family has density $\frac{1}{2}\sech(\frac{\pi}{2}x)dx$, and may be thought of as a smoothed Laplace distribution.}
    \label{tab:nef-qvf}
\end{table}

\begin{proposition}
    Any NEF-QVF can be obtained by an affine transformation ($X \mapsto aX + b$ applied to the underlying random variables) of one of the six families listed in Table~\ref{tab:nef-qvf}.
    Conversely, any affine transformation of an NEF-QVF yields another NEF-QVF.
\end{proposition}

We will study \emph{hypothesis testing} in high-dimensional products of NEF-QVFs.
That is, we seek to distinguish with high probability two sequences of distributions: the \emph{null} distributions $\QQ_n$, and the \emph{planted} or \emph{alternative} distributions $\PP_n$.
We consider two possible relationships between the two sequences: (1) \emph{kin spiking}, where $\PP_n$ belongs to the same NEF as $\QQ_n$ but has a different mean, and (2) \emph{additive spiking}, where $\PP_n$ is a translation of $\QQ_n$, possibly not belonging to the same NEF.
Kin spiking will be mathematically more elegant, but additive spiking will allow us to treat a spiked matrix model similar to those studied in earlier works.
\begin{definition}[Kin-spiked NEF-QVF model]
    \label{def:kin-spiked}
    Let $\widetilde{\rho}_{\mu}$ be a mean-parametrized NEF-QVF over $\mu \in \Omega \subseteq \RR$.\footnote{It will not matter for our purposes what the base measure $\rho_0$ is.}
    Let $N = N(n) \in \NN$ and $\mu_{n, i} \in \Omega$ for each $n \in \NN$ and $i \in [N(n)]$.
    Let $\sP_n$ be a probability measure over $\Omega^{N(n)}$.
    Then, define sequences of probability measures $\PP_n, \QQ_n$ as follows:
    \begin{itemize}
    \item Under $\QQ_n$, draw $y_i \sim \widetilde{\rho}_{\mu_{n, i}}$ independently for $i \in [N(n)]$.
    \item Under $\PP_n$, draw $\bx \sim \sP_n$, and then draw $y_i \sim \widetilde{\rho}_{x_i}$ independently for $i \in [N(n)]$.
    \end{itemize}
\end{definition}

\begin{definition}[Additively-spiked NEF-QVF model]
    \label{def:additively-spiked}
    In the same setting as Definition~\ref{def:kin-spiked} but with $\sP_n$ now a probability measure over $\RR^{N(n)}$, define:
    \begin{itemize}
    \item Under $\QQ_n$, draw $y_i \sim \widetilde{\rho}_{\mu_{n, i}}$ independently for $i \in [N(n)]$ (the same as above).
    \item Under $\PP_n$, first draw $\bx \sim \sP_n$ and $z_i \sim \widetilde{\rho}_{\mu_{n, i}}$ independently for $i \in [N(n)]$, and observe $y_i = x_i + z_i$.
    \end{itemize}
\end{definition}

In either case, we will focus on the problem of \emph{strong detection}: given a particular time budget $T(n)$, does there exist a sequence of tests $f_n: \RR^{N(n)} \to \{\mathtt{p}, \mathtt{q}\}$ computable in time $T(n)$ and having
\begin{equation}
    \label{eq:test-success}
    \lim_{n \to \infty} \QQ_n[f_n(\by) = \texttt{q} ] = \lim_{n \to \infty} \PP_n[f_n(\by) = \texttt{p} ] = 1 \, \text{?}
\end{equation}
\noindent
To the best of our knowledge, with the exception of negatively-spiked Gaussian Wishart models \cite{BKW-2019-ConstrainedPCA,BBKMW-2020-SpectralPlantingColoring}, all previous applications of the low-degree method to hypothesis testing in the literature may be expressed as kin-spiked models in the Gaussian or Bernoulli NEF-QVFs.\footnote{We remark that kin and additive spiking are equivalent in the Gaussian NEF-QVF.}

\subsection{The low-degree likelihood ratio method}
\label{sec:low-deg}

We now describe the calculation, based on the question \eqref{eq:intro-low-deg} mentioned earlier, that we will take as a heuristic proxy for the difficulty of hypothesis testing.
This will suggest that the problem of efficient strong detection described above may be addressed with computations involving the likelihood ratio,
\begin{equation}
    L_n(\by) \colonequals \frac{d\PP_n}{d\QQ_n}(\by).
\end{equation}
These techniques work in the Hilbert space $L^2(\QQ_n)$, having inner product $\langle f, g \rangle \colonequals \EE_{\by \sim \QQ_n}[f(\by)g(\by)]$ and associated norm $\|f\|^2 \colonequals \langle f, f \rangle$.

The following is the basic statement relating \eqref{eq:intro-low-deg} to the likelihood ratio, which follows from a linear-algebraic calculation.
\begin{proposition}[\cite{HS-2017-BayesianEstimation,HKPRSS-2017-SOSSpectral,Hopkins-2018-Thesis}]
    \label{prop:low-deg-opt}
    Denote by $L_n^{\leq D}$ the orthogonal projection of $L_n$ to the subspace of $L^2(\QQ_n)$ consisting of polynomials having degree at most $D$.
    Then,
    \begin{equation}
        \left\{\begin{array}{ll} \text{maximize} & \EE_{\by \sim \PP_n} f_n(\by) \\[0.5em] \text{subject to} & f_n \in \RR[\by]_{\leq D} \cap L^2(\QQ_n) \\[0.5em] & \EE_{\by \sim \QQ_n} f_n(\by)^2 = 1 \end{array}\right\} = \|L_n^{\leq D}\|,
    \end{equation}
    with optimizer $f_n^{\star} = L_n^{\leq D} / \|L_n^{\leq D}\|$.
\end{proposition}
\noindent
We remark that if we took $D = +\infty$ here and optimized over all of $L^2(\QQ_n)$, then the optimizer, again by a straightforward calculation, would simply be the likelihood ratio $f_n^{\star} = L_n / \|L_n\|$, and the optimal value its norm $\|L_n\|$.
The likelihood ratio is closely related to optimal hypothesis testing via the Neyman-Pearson lemma \cite{NP-1933-MostEfficientTests}, and the norm $\|L_n\|$ also plays a role through Le Cam's notion of \emph{contiguity} and the associated \emph{second moment method} \cite{LCY-2012-AsymptoticsStatistics}.
In particular, if $\|L_n\|$ is bounded as $n \to \infty$, then no test can achieve strong detection, a fact that has been used to great effect for several high-dimensional problems in recent literature \cite{MRZ-2015-LimitationsSpectral,BMVVX-2018-InfoTheoretic,PWBM-2018-PCAI}.

The ``truncated'' variant $L_n^{\leq D}$ is called the \emph{low-degree likelihood ratio}, and emerges according to the above result as the main object conjecturally controlling the computational cost of hypothesis testing.
The basic conjecture concerning the norm of the low-degree likelihood ratio is as follows.
\begin{conjecture}[Informal; Conjecture 2.2.4 of \cite{Hopkins-2018-Thesis} and Conjecture 1.16 of \cite{KWB-2019-NotesLowDegree}]
    \label{conj:low-deg}
    For ``sufficiently nice'' sequences of probability measures $\PP_n$ and $\QQ_n$, if there exists $\epsilon > 0$, $D = D(n) \geq (\log n)^{1 + \epsilon}$, and a constant $K$ such that $\|L_n^{\leq D}\| \leq K$ for all $n$, then there is no sequence of $f_n$ that are computable in polynomial time in $n$ and that achieve the strong detection conditions \eqref{eq:test-success}.
\end{conjecture}
\noindent
Different scalings of $D(n)$ are also conjectured to capture hardness of testing with various other computational time budgets ranging from polynomial to exponential (see, e.g., Section~3 of \cite{KWB-2019-NotesLowDegree} and \cite{DKWB-2019-SubexponentialTimeSparsePCA,DKWB-2020-AverageCaseRIP}); roughly speaking, if $\|L_n^{\leq D(n)}\|$ is bounded, then we expect strong detection in time $T(n) = \exp(D(n))$ to be impossible.

These preliminaries established, we may outline our contributions at a high level: we will compute and bound $\|L_n^{\leq D}\|$ for kin-spiked and additively-spiked NEF-QVF models, and compare these results with Conjecture~\ref{conj:low-deg} and variants thereof to produce new predictions and challenges for the low-degree method.

\section{Main results}

\subsection{Low-degree analysis and $z$-score overlap}
Our first result is a general bound on $\|L_n^{\leq D}\|$ in kin-spiked NEF-QVF models.
Before presenting the statement, we highlight a previously-known special case that the general result will resemble: for the NEF-QVF of Gaussian distributions with unit variance, \cite{KWB-2019-NotesLowDegree} showed the following formula:
\begin{equation}
    \|L_n^{\leq D}\|^2  = \Ex_{\bx^1, \bx^2} \exp^{\leq D}(\langle \bx^1, \bx^2 \rangle),
\end{equation}
where $\bx^i$ are drawn independently from $\sP_n$ and $\exp^{\leq D}(t) = \sum_{k = 0}^D t^k / k!$ is the order-$D$ Taylor expansion of the exponential function.
This formula is both elegant in principle, showing that the behavior of $\|L_n^{\leq D}\|$ may be reduced to the behavior of the scalar ``overlap'' random variables $\langle \bx^1, \bx^2 \rangle$, and leads to simplified proofs of low-degree analyses in practice.
We show that a similar result holds in any NEF-QVF, so long as we (1)~replace the equality with a suitable inequality, (2)~replace $\bx^i$ with suitably centered and normalized $z$-scores, and (3)~replace the exponential function with a suitable relative, which will depend on the NEF-QVF's value of $v_2$, the coefficient of $\mu^2$ in the variance function.

We first briefly describe the set of all possible values of $v_2$.
\begin{proposition}
    Let $\mathcal{V} \colonequals [0, +\infty) \cup \{-\frac{1}{m}: m \in \ZZ_{\geq 1}\} \subset \RR$.
    Then, for any NEF-QVF, $v_2 \in \mathcal{V}$.
    Conversely, for any $v \in \mathcal{V}$, there exists an NEF-QVF with $v_2 = v$.
    The only NEF-QVFs with $v_2 < 0$ are the binomial families (including Bernoulli), and the only NEF-QVFs with $v_2 = 0$ are the Gaussian and Poisson families.
\end{proposition}

\begin{definition}
    \label{def:f}
    For $t \in \RR$ and $v \in \sV$, define
    \begin{equation}
        f(t; v) \colonequals \left\{\begin{array}{ll} e^t & \text{if } v = 0, \\ (1 - vt)^{-1 / v} & \text{if } v \neq 0 \text{ and } t < 1 / |v|, \\ +\infty & \text{if } v > 0 \text{ and } t \geq 1 / |v|. \end{array}\right.
    \end{equation}
    Moreover, for $D \in \NN$, let $f^{\leq D}(t; v)$ denote the order-$D$ Taylor expansion of $f(t; v)$ about $t = 0$ for fixed $v$, and let $f^{\leq +\infty}(t; v) \colonequals f(t; v)$.
\end{definition}
\noindent
See Figure~\ref{fig:fs} for an illustration of these functions ``sandwiching'' the exponential.

\begin{figure}
    \begin{center}
        \includegraphics[scale=0.8]{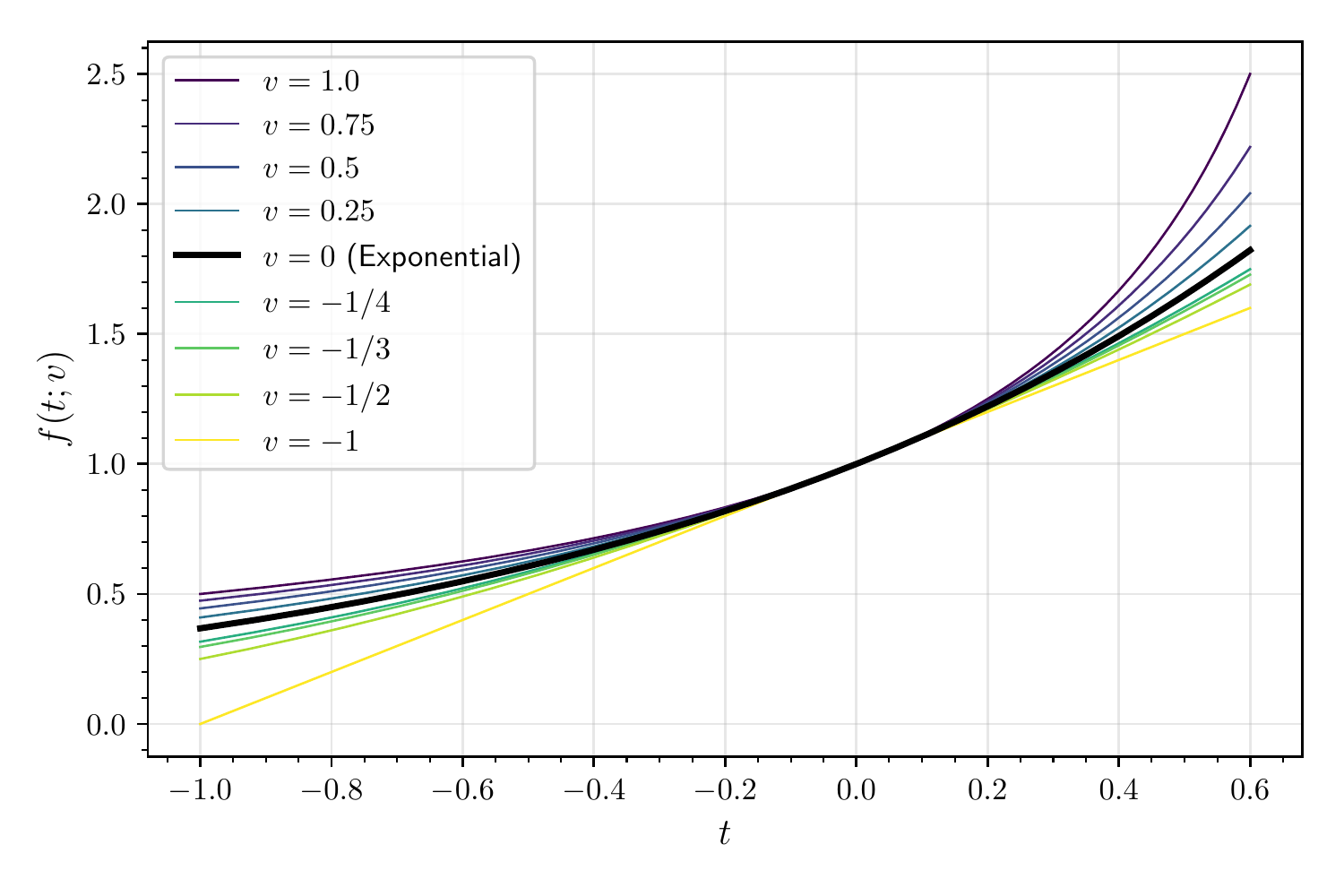}
    \end{center}
    \vspace{-2em}
    \caption{We plot $f(t; v)$ near $t = 0$ for various values of $v$, emphasizing the monotonicity in $v$ and the appearance of the exponential function for $v = 0$.}
    \label{fig:fs}
\end{figure}

\begin{theorem}
    \label{thm:trunc-lr}
    Let $\widetilde{\rho}_{\mu}$ be a mean-parametrized NEF-QVF over $\mu \in \Omega \subseteq \RR$, with variance function $V(\mu) = v_0 + v_1\mu + v_2\mu^2$.
    For $\mu, x \in \Omega$, define the \emph{$z$-score} as
    \begin{equation}
        z_{\mu}(x) \colonequals \frac{x - \mu}{\sqrt{V(\mu)}}.
    \end{equation}
    Let $\mu_{n, i} \in \Omega$ and $\sP_n$ be as in Definition~\ref{def:kin-spiked} of the kin-spiked NEF-QVF model.
    Define the \emph{$z$-score overlap},
    \begin{equation}
        r_n \colonequals \sum_{i = 1}^{N(n)} z_{\mu_{n, i}}(x_i^1)z_{\mu_{n, i}}(x_i^2),
    \end{equation}
    where $\bx^1, \bx^2 \sim \mathcal{P}_n$ independently.
    Let $L_n^{\leq D}$ denote the low-degree likelihood ratio.
    \begin{itemize}
        \item If $v_2 \geq 0$, then for any $n \in \NN$ and $D \in \NN \cup \{+ \infty\}$,
            \begin{equation}
        \label{eq:trunc-lr-pos}
        \|L_n^{\leq D}\|^2 \leq \EE\left[ f^{\leq D}(r_n; v_2)\right],
        \end{equation}
    and equality holds if $v_2 = 0$ (i.e., in the Gaussian and Poisson NEFs).
\item If $v_2 < 0$, then for any $n \in \NN$ and $D \in \NN \cup \{+ \infty\}$,
    \begin{equation}
        \label{eq:trunc-lr-neg}
        \EE\left[ f^{\leq D}(r_n; v_2)\right] \leq \|L_n^{\leq D}\|^2 \leq \EE\left[ f^{\leq D}(r_n; 0)\right].
    \end{equation}
\end{itemize}
\end{theorem}

We also give more cumbersome exact formulae in Section~\ref{sec:low-deg-nefqvf}, but  emphasize these bounds here for their similarity to the simpler Gaussian case.
In particular, since $f(t; v) \approx \exp(t)$ near $t = 0$ for any $v$, to a first approximation it appears reasonable to estimate $\|L_n^{\leq D}\|^2 \approx \EE[\exp^{\leq D}(r_n)]$.
We suggest this as a powerful heuristic, far simpler than the full low-degree likelihood ratio analysis, for making quick predictions of computational thresholds.
As an example, we use this to informally derive the Kesten-Stigum threshold in the symmetric stochastic block model with two communities in Section~\ref{sec:kesten-stigum}, in merely a few lines of calculation requiring no graph-theoretic intuition.

\subsection{Channel monotonicity}
The monotonicity of the functions $f(t; v)$ in $v$ evident in Figure~\ref{fig:fs} suggests that we might expect $\|L_n^{\leq D}\|$ to be monotone across different kin-spiked NEF-QVF models with the same mean distribution $\sP_n$.
While this does not follow directly from the above result, a slightly more careful argument shows that it is indeed the case.
\begin{theorem}
    \label{thm:comparison}
    Suppose $L_n^{(i)}$ for $i \in \{1, 2\}$ are the likelihood ratios for the hypothesis testing problems in two kin-spiked NEF-QVF models, with mean domains $\Omega^{(i)}$ and variance functions $V^{(i)}(\mu) = v_0^{(i)} + v_1^{(i)}\mu + v_2^{(i)}\mu^2$.
    Suppose that the null means $\mu_{n, j}$ and the distribution $\mathcal{P}_n$ are the same in both problems (in particular, $\Omega^{(1)} \cap \Omega^{(2)}$ must contain the support of $\mathcal{P}_n$).
    If $v_2^{(1)} \leq v_2^{(2)}$, then, for any $D \in \NN \cup \{+ \infty\}$, $\|(L_n^{(1)})^{\leq D}\|^2 \leq \|(L_n^{(2)})^{\leq D}\|^2$.
\end{theorem}
\noindent
Informally, this says that if $v_2^{(1)} \leq v_2^{(2)}$, then ``Problem 1 is at least as hard as Problem 2,'' for any given computational budget.
For example, for a fixed collection of null means $\mu_{n, i}$ and a fixed spike mean distribution $\sP_n$, we would predict the following relationships among output ``channels'' or observation distributions, with ``$\geq$'' denoting greater computational difficulty:
\begin{equation}
    \text{Bernoulli} \geq \text{Binomial} \geq \text{Gaussian} = \text{Poisson} \geq \text{Exponential}.
\end{equation}

This suggests two intriguing open problems further probing the low degree method.
First, are these predictions in fact accurate, i.e., can they be corroborated with any other form of evidence of computational hardness?
(One intriguing possibility is average-case reductions in the style of \cite{BR-2013-SparsePCA, BB-2020-ReducibilityStatCompGaps} between different NEF-QVFs.)
And second, if these predictions are accurate, then does \emph{strict} inequality hold in computational cost between any of these versions of a given problem, or does \emph{channel universality} hold (we borrow the term from \cite{LKZ-2015-LowRankChannelUniversality} but use it in a slightly different sense), where in fact computational complexity of testing does not depend on the NEF-QVF through which the data are observed?

\subsection{Hyperbolic secant spiked matrix model}
We now turn our attention to additively-spiked NEF-QVF models, and isolate one particular model where we can both perform explicit calculations and draw a comparison to prior works.
This is a \emph{spiked matrix model}---asking us to determine whether an unstructured random matrix has been deformed by a rank-one perturbation---with noise distributed according to $\rho^{\sech}$ the probability measure on $\RR$ which has the following density $w(x)$ with respect to Lebesgue measure:
\begin{equation}
    w(x) \colonequals \frac{1}{2\cosh(\pi x / 2)} = \frac{1}{2}\sech(\pi x / 2).
\end{equation}
This density belongs to the rather obscure class of ``generalized hyperbolic secant'' NEFs mentioned in Table~\ref{tab:nef-qvf}.
It may be viewed as a smoothing of the Laplace distribution; see Figure~\ref{fig:sech}.\footnote{This density has some other remarkable mathematical properties: (1) like the Gaussian density, up to dilation $w(x)$ is its own Fourier transform, and (2) $w(x)$ is the Poisson kernel over the strip $\{z: \Im(z) \in [-1, 1]\} \subset \CC$.}

\begin{figure}
    \begin{center}
        \includegraphics[scale=0.75]{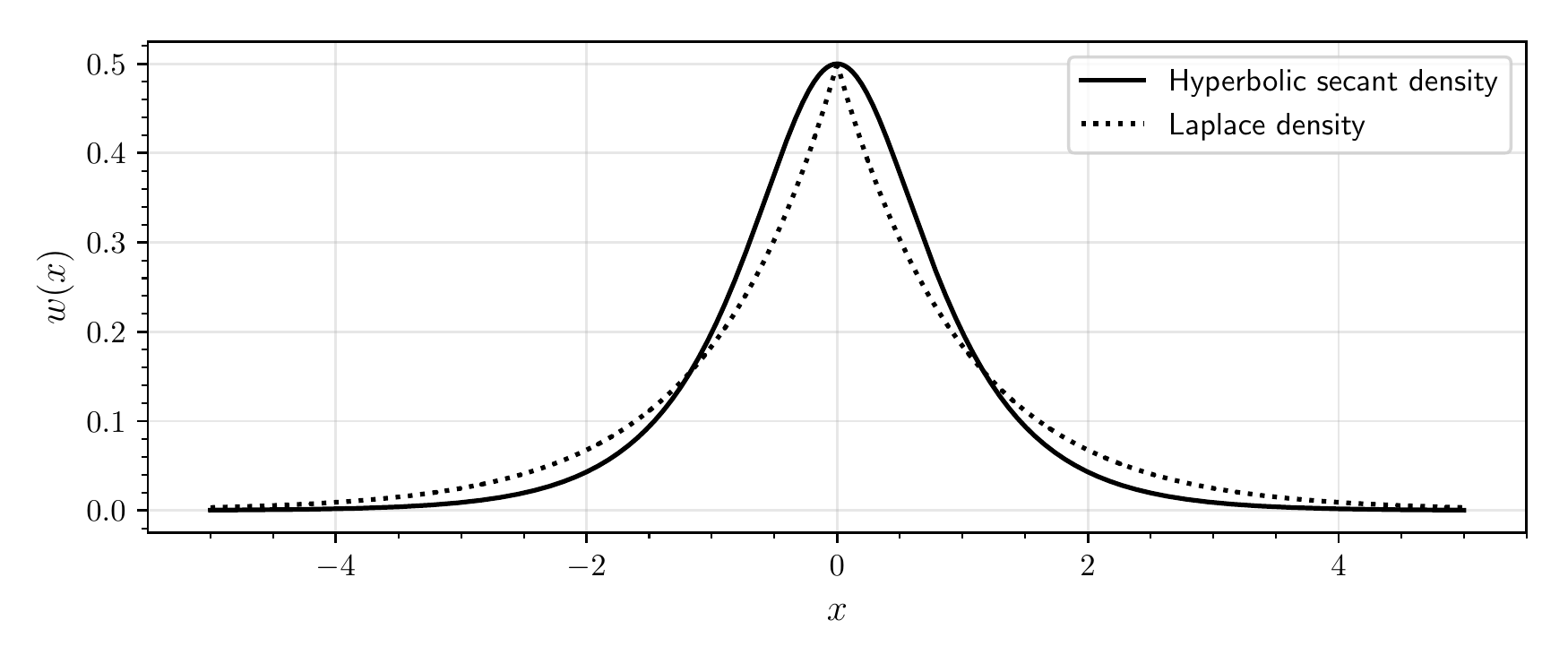}
    \end{center}
    \vspace{-2em}
    \caption{We plot the density $w(x)$ of the hyperbolic secant distribution used in Theorem~\ref{thm:spiked-mx}, showing that it is effectively a smoothed version of the density $\frac{1}{2}\exp(-|x|)$ of the better-known Laplace distribution.}
    \label{fig:sech}
\end{figure}

We next specify the spiked matrix model we will study.
\begin{definition}[Rademacher-spiked Wigner matrix models]
    Given a probability measure $\rho$ over $\RR$, we write $\Wig(\rho, \lambda) = ((\QQ_n, \PP_n))_{n = 1}^{\infty}$ for the sequence of pairs of probability measures $\QQ_n$ and $\PP_n$ over $\RR^{\binom{[n]}{2}}$ defined as follows:
    \begin{itemize}
    \item Under $\bY \sim \QQ_n$, we draw $Y_{\{i, j\}} \sim \rho$ independently for all $i < j$.
    \item Under $\bY \sim \PP_n$, we first draw $\bx \sim \Unif(\{\pm 1\}^n)$ and $Z_{\{i, j\}} \sim \rho$ independently for all $i < j$, and then set $Y_{\{i, j\}} = \frac{\lambda}{\sqrt{n}} x_ix_j + Z_{\{i, j\}}$.
    \end{itemize}
\end{definition}
\noindent
Thus $\Wig(\rho^{\sech}, \lambda)$ is an additively-spiked NEF-QVF model, as described in Definition~\ref{def:additively-spiked}.
We may also view $\bY = \frac{\lambda}{\sqrt{n}} \bx\bx^{\top} + \bZ$ as symmetric matrices, where we omit the diagonal from the observations (this is a technical convenience that is straightforward but tedious to eliminate).

The following result characterizes two testing algorithms related to computing the largest eigenvalue.
For $\bY$ itself, for sufficiently large $\lambda$, the largest eigenvalue undergoes a ``pushout'' effect under $\PP_n$ and becomes larger than the typical largest eigenvalue under $\QQ_n$, as characterized by a variant of the Baik--Ben Arous--P\'{e}ch\'{e} (BBP) transition of random matrix theory \cite{BBAP-2005-LargestEigenvalueSampleCovariance,CDMF-2009-DeformedWigner}.
It turns out, however, that it is suboptimal to merely compute and threshold the largest eigenvalue of $\bY$; instead, the optimal algorithm is to first apply an entrywise transformation and only then compute and threshold the largest eigenvalue.
\begin{proposition}[Better-than-BBP testing \cite{CDMF-2009-DeformedWigner,PWBM-2018-PCAI}]
    \label{prop:better-than-bbp}
    Define $\lambda_* \colonequals 2\sqrt{2} / \pi \approx 0.9$.
    \begin{itemize}
    \item If $\lambda > 1$, then strong detection in $\Wig(\rho^{\sech}, \lambda)$ can be achieved in polynomial time by the \emph{PCA test},
        \begin{equation}
            f^{\mathsf{PCA}}(\bY) \colonequals \left\{\begin{array}{ll} \texttt{p} & \text{if } \frac{1}{\sqrt{n}}\lambda_{\max}(\bY) \geq \frac{1}{2}(2 + \lambda + \lambda^{-1}), \\ \texttt{q} & \text{otherwise.}\end{array}\right.
        \end{equation}
    \item If $\lambda < 1$, then $f^{\mathsf{PCA}}$ does not achieve strong detection in $\Wig(\rho^{\sech}, \lambda)$.
    \item If $\lambda > \lambda_*$, then strong detection in $\Wig(\rho^{\sech}, \lambda)$ can be achieved in polynomial time by the \emph{pre-transformed PCA test}:
        \begin{equation}
            f^{\mathsf{tPCA}}(\bY) \colonequals \left\{\begin{array}{ll} \texttt{p} & \text{if } \frac{1}{\sqrt{n}}\lambda_{\max}\left(\frac{\pi}{2}\tanh\left(\frac{\pi}{2}\bY\right)\right) \geq \frac{1}{2}(2\lambda_* + \lambda_*^2\cdot \lambda + \lambda^{-1}), \\ \texttt{q} & \text{otherwise.} \end{array}\right.
        \end{equation}
        Here, $\tanh(\cdot)$ is applied entrywise to the matrix argument.
    \item If $\lambda < \lambda_*$, then there exists no test (efficiently computable or not) that can achieve strong detection in $\Wig(\rho^{\sech}, \lambda)$.
    \end{itemize}
\end{proposition}
\noindent
The threshold $\lambda_*$ is related to the \emph{Fisher information} in the family of translates of $\rho^{\sech}$ as $\lambda_* = (\int_{-\infty}^{\infty}w^{\prime}(x)^2 / w(x) dx)^{-1/2}$, and the optimal entrywise transformation is the logarithmic derivative $\frac{\pi}{2}\tanh(\frac{\pi}{2}x) = -w^{\prime}(x) / w(x)$; the results of \cite{PWBM-2018-PCAI} show that both relationships hold for optimal tests in non-Gaussian spiked matrix models in great generality.

While low-degree polynomials can approximate the test $f^{\mathsf{PCA}}$ via the power method, the transcendental entrywise $\tanh(\cdot)$ transformation used by $f^{\mathsf{tPCA}}$ seems rather ill-suited to the low-degree analysis.
We show below that, indeed, if we attempt to carry out the low-degree method for this problem while bounding the \emph{entrywise degree} of the polynomials involved---the greatest power with which any given entry of $\bY$ can appear---then we obtain an incorrect threshold.
Loosely speaking, this suggests that some analytic computation like the transcendental $\tanh(\cdot)$ operation is in fact \emph{necessary} to obtain an optimal test.

\begin{definition}[Entrywise degree]
    For a polynomial $p \in \RR[y_1, \dots, y_N]$, write $\deg_i(p)$ for the greatest power with which $y_i$ occurs in a monomial having non-zero coefficient in $p$.
\end{definition}

\begin{theorem}
    \label{thm:spiked-mx}
    Suppose $D \in \NN$ and $0 < \lambda < \lambda_* + \frac{1}{20D}$.
    Then,
    \begin{equation}
        \limsup_{n \to \infty} \left\{\begin{array}{ll} \text{maximize} & \EE_{\bY \sim \PP_n} f_n(\bY) \\[0.5em] \text{subject to} & f_n \in \RR[\bY], \\[0.5em] & \deg_{\{i, j\}}(f_n) \leq D \text{ for all } \{i, j\} \in \binom{[n]}{2}, \\[0.5em] & \EE_{\bY \sim \QQ_n} f_n(\bY)^2 = 1 \end{array}\right\} < +\infty.
    \end{equation}
\end{theorem}
\noindent
That is, when we restrict our attention to polynomials of entrywise degree at most $D$ a constant not growing with $n$, the apparent computational threshold suggested by the corresponding low-degree calculation shifts by $\Omega(1 / D)$ from the true value.

This limitation applies, for example, to the approach of the recent paper \cite{DHS-2020-SpikedMatrixHeavyTailed}.
The authors propose to build tests and estimators for spiked matrix models that remain effective under heavy-tailed noise distributions by using polynomials that sum over monomials indexed by \emph{self-avoiding walks} on the matrix $\bY$.
In particular, they show that, for $\lambda > 1$---the optimal threshold for Gaussian noise---such polynomials can successfully achieve strong detection in $\Wig(\rho, \lambda)$ for a wide variety of measures $\rho$, ranging from Gaussian $\rho$ to very heavy-tailed $\rho$ for which $f^{\mathsf{PCA}}$ fails severely.
However, our result implies that, since these polynomials have entrywise degree 1 (that is, they are multilinear), such polynomials (and many generalizations thereof to higher but bounded entrywise degree) \emph{cannot} achieve strong detection for all $\lambda > \lambda_*$, and thus are suboptimal for this model.

\subsection{Exploration before inference: a challenge for low-degree polynomials}

The discussion above suggests that, for algorithms computing low-degree polynomials, there is a tension between robustness to heavy-tailed noise distributions and optimality for specific rapidly-decaying (and, in the case above, non-Gaussian) noise distributions.
We propose the following hypothesis testing problem to capture a simple case of this challenge.

\begin{definition}[Mixed spiked matrix model]
    Fix $\alpha > 1$ and $\lambda > 0$.
    Let $\rho^{\heavy}$ be the probability measure over $\RR$ with density proportional to $(1 + x^2)^{-\alpha / 2}$.
    Denote $\Wig(\rho^{\sech}, \lambda) \equalscolon ((\QQ^{(1)}_n, \PP^{(1)}_n))_{n = 1}^{\infty}$ and $\Wig(\rho^{\heavy}, \lambda) \equalscolon ((\QQ^{(2)}_n, \PP^{(2)}_n))_{n = 1}^{\infty}$.
    Let $\Mix(\alpha, \lambda) = ((\QQ_n, \PP_n))_{n = 1}^{\infty}$ for $\QQ_n$ and $\PP_n$ defined as follows:
    \begin{itemize}
    \item Under $\QQ_n$, draw $i \sim \Unif(\{1, 2\})$, and observe $\bY \sim \QQ_n^{(i)}$.
    \item Under $\PP_n$, observe $\bY \sim \PP_n^{(1)}$.
    \end{itemize}
\end{definition}
\noindent
In words, in the null distribution we flip a coin to choose between a heavy-tailed and a rapidly-decaying but non-Gaussian entrywise distribution for the noise matrix, while in the planted distribution we always observe the latter kind of noise.

The motivation for this definition should be clear at a technical level in the context of the low-degree likelihood ratio analysis: the possibility of receiving heavy-tailed inputs in effect restricts the available polynomials to low entrywise degree, which in turn precludes optimal testing in $\Wig(\rho^{\sech}, \lambda)$.
Still, we suggest that this example is not as artificial as it might appear; on the contrary, it captures an important reality of statistical practice.
When faced with a dataset, the statistician knows to first examine the data, and in particular assess what distributional assumptions could be justified, before applying a particular algorithm.
We might, for instance, think of the case $i = 2$ in $\QQ_n$ as observing the results of a severely miscalibrated experiment, in which case we should not use an inference procedure fine-tuned to our distributional expectations.
As this initial examination is often said to fall under the rubric of ``exploratory data analysis,'' we call this algorithmic strategy \emph{exploration before inference}.

It appears difficult for low-degree polynomials to match the performance of algorithms performing exploration before inference.
In particular, we have the following result.
\begin{theorem}
    \label{thm:mixed-spiked-mx}
    For all $\alpha > 1$ and $\lambda > \lambda_*$, there is a polynomial-time algorithm achieving strong detection in $\Mix(\alpha, \lambda)$.
    However, for all $0 < \lambda < \lambda_* + \frac{1}{10\alpha}$, we have
    \begin{equation}
        \limsup_{n \to \infty}\left\{\begin{array}{ll} \text{maximize} & \EE_{\bY \sim \PP_n} f_n(\bY) \\[0.5em] \text{subject to} & f_n \in \RR[\bY] \cap L^2(\QQ_n) \\[0.5em] & \EE_{\bY \sim \QQ_n} f_n(\bY)^2 = 1 \end{array}\right\} < +\infty.
    \end{equation}
\end{theorem}
\noindent
We remark that the optimization above is over \emph{all} polynomials, with no degree constraint; the only constraint is that the polynomials belong to $L^2(\QQ_n)$.
While we have specifically engineered $\Mix(\alpha, \lambda)$ to have fewer polynomials in $L^2(\QQ_n)$, it seems reasonable to conjecture that even algorithms that threshold polynomials in the sense of \eqref{eq:intro-poly-thresh} would fail to achieve strong detection in this model.
On the other hand, the algorithm achieving the first statement is a simple application of exploration before inference: first, it examines the entrywise maximum of $\bY$ to determine with high probability whether $i = 2$ was chosen under $\QQ_n$.
If so, it returns $\texttt{q}$; if not, it applies the test $f^{\mathsf{tPCA}}$ from Proposition~\ref{prop:better-than-bbp}.

One possible resolution of this difficulty is the notion of \emph{coordinate degree} suggested in the formalization of the low-degree method in Hopkins' thesis \cite{Hopkins-2018-Thesis}.
Functions of low coordinate degree are those spanned by functions depending only on a small number of variables; this gives a generalization of low-degree polynomials that is more natural in that it does not depend upon the specific functional basis of low-degree monomials.
At the very least, $\Mix(\cdot, \cdot)$ gives a fairly natural model where optimizing over functions of low coordinate degree rather than low-degree polynomials is necessary for the low-degree method to make correct predictions.

It is an intriguing open problem whether functions of low coordinate degree can, in fact, achieve strong detection in $\Mix(\alpha, \lambda)$ for all $\alpha > 1$ and $\lambda > \lambda_*$ (as well as various other caricatures of situations requiring exploration before inference).
It seems difficult to encode the ``branching'' or ``if statement'' computational step of the algorithm described above into a function of low coordinate degree, and thus plausible that the low-degree method, even augmented with the notion of coordinate degree, might make an incorrect prediction for $\Mix(\alpha, \lambda)$ and similar models.

We note also that exploration before inference need not always coincide with \emph{robustness}, where we want an algorithm to perform well under poorly-behaved random or adversarial corruptions of the inputs.
Exploration before inference could also arise in, e.g., a model choosing one of several different rapidly-decaying noise distributions, where an algorithm might estimate the noise distribution before performing inference.
This type of setting still seems difficult for low-degree polynomials, though they might appear successful in averaged formulations in the style of Proposition~\ref{prop:low-deg-opt} by simply succeeding on one of the noise distributions and neglecting the others.

\section{Low-degree likelihood ratio analysis in NEF-QVFs}
\label{sec:low-deg-nefqvf}

In this section, we give an explicit formulae for $\|L_n^{\leq D}\|^2$ in kin- and additively-spiked NEF-QVF models, and prove the bounds in Theorem~\ref{thm:trunc-lr}.
Our strategy will be to decompose $L_n$ according to orthogonal polynomials in $L^2(\QQ_n)$, and sum the masses of the components of low-degree polynomials.

Therefore, in Section~\ref{sec:orth-poly}, we first review the general description of orthogonal polynomials in NEF-QVFs, and prove some minor further results that will be useful.
Then, in Section~\ref{sec:components} we give formulae for the coefficients of each orthogonal polynomial component of the likelihood ratio in both the kin- and additively-spiked models.
Finally, in Sections~\ref{sec:full-norm} and \ref{sec:trunc-norm}, we prove results concerning the norms of the full and low-degree likelihood ratios in kin-spiked models (including Theorem~\ref{thm:trunc-lr}), where further simplifications are possible.

\subsection{Orthogonal polynomials in NEF-QVFs}
\label{sec:orth-poly}

Our main tool will be that, in NEF-QVFs, there is a remarkable connection between the likelihood ratio and the orthogonal polynomials of $\rho_\theta$.
The likelihood ratio in any NEF is simple:
\begin{equation}
    L(y; \theta) \colonequals \frac{d\rho_\theta}{d\rho_0}(y) = \exp(y\theta - \psi(\theta)),
\end{equation}
where $\psi(\theta) = \EE_{x \sim \rho_0}[\exp(\theta x)]$.
We may also reparametrize in terms of the mean:
\begin{equation}
    \widetilde{L}(y; \mu) \colonequals L(y; (\psi^\prime)^{-1}(\mu)) = \exp(y (\psi^\prime)^{-1}(\mu) - \psi((\psi^\prime)^{-1}(\mu))).
\end{equation}
As the following result of Morris shows, in an NEF-QVF, $\widetilde{L}(y; \mu)$ is a kind of generating function of the orthogonal polynomials of $\widetilde{\rho}_\mu$.
\begin{definition}
    \label{def:a-consts}
    For $v \in \RR$, define the sequences of constants
    \begin{align}
      \what{a}_k(v) &\colonequals \prod_{j = 0}^{k - 1}(1 + vj), \label{eq:what-a} \\
      a_k(v) &\colonequals k! \cdot \what{a}_k(v).
    \end{align}
\end{definition}
\begin{proposition}[NEF-QVF Rodrigues Formula; Theorem 4 of \cite{Morris-1982-NEFQVF}]
    \label{prop:nef-qvf-rodrigues}
    Let $\mu_0 = \psi^{\prime}(0) = \EE_{x \sim \rho_0}[x]$.
    Define the polynomials
    \begin{equation}
        p_k(y; \mu_0) \colonequals \frac{V(\mu_0)^{k}}{\widetilde{L}(y, \mu_0)} \cdot \frac{d^k\widetilde{L}}{d\mu^k}(y, \mu_0).
    \end{equation}
    Then, $p_k(y; \mu_0)$ is a degree $k$ monic polynomial in $y$, and this family satisfies the orthogonality relation
    \begin{equation}
        \Ex_{y \sim \widetilde{\rho}_{\mu_0}} p_k(y; \mu_0)p_\ell(y; \mu_0) = \delta_{k\ell} \cdot a_k(v_2) V(\mu_0)^k.
    \end{equation}
    In particular, defining the normalized polynomials
    \begin{equation}
        \what{p}_k(y; \mu_0) \colonequals \frac{1}{V(\mu_0)^{k / 2}\sqrt{a_k(v_2)}}p_k(y; \mu_0),
    \end{equation}
    the $\what{p}_k(y; \mu_0)$ are orthonormal polynomials for $\widetilde{\rho}_{\mu_0}$.
\end{proposition}

The main property of these polynomials that will be useful for us is the following identity, also obtained by Morris, giving the expectation of a given orthogonal polynomial under the kin spiking operation, i.e., under a different distribution from the same NEF-QVF.
\begin{proposition}[Corollary 1 of \cite{Morris-1982-NEFQVF}]
    For all $k \in \NN$ and $x, \mu \in \Omega$,
    \begin{equation}
        \Ex_{y \sim \widetilde{\rho}_{x}} p_k(y; \mu) = \what{a}_k(v_2) (x - \mu)^k.
    \end{equation}
\end{proposition}
\noindent
We may obtain a straightforward further corollary by including the normalization, which allows us to incorporate the variance factor into a $z$-score, as follows.
\begin{corollary}[Kin-spiked expectation]
    \label{cor:mismatched-mean}
    For all $k \in \NN$ and $x, \mu \in \Omega$,
    \[ \Ex_{y \sim \widetilde{\rho}_{x}} \what{p}_k(y; \mu) = \sqrt{\frac{\what{a}_k(v_2)}{k!}} z_{\mu}(x)^k. \]
\end{corollary}

We will also be interested in the analogous result for additive spiking.
This result is less elegant, and is expressed in terms of another polynomial sequence.
First, we write the precise generating function relation between the likelihood ratio and the orthogonal polynomials.
\begin{proposition}[Generating function]
    Let $\mu \in \Omega$ and write $\psi(\eta) = \EE_{x \sim \widetilde{\rho}_{\mu}}[\exp(\eta x)]$.
    Then,
    \begin{align}
      \sum_{k \geq 0} \frac{z_{\mu}(t)^k}{k!}p(y; \mu) = \exp\big(y (\psi^\prime)^{-1}(t) - \psi((\psi^\prime)^{-1}(t))\big). \label{eq:op-gen-fn}
    \end{align}
\end{proposition}
\noindent
Note that here we are ``rebasing'' the NEF-QVF to have $\widetilde{\rho}_{\mu}$ as the base measure by our definition of $\psi(\cdot)$.
One may view this result as generalizing to NEF-QVFs the generating function $\exp(ty - \frac{1}{2}t^2)$ for Hermite polynomials.
The key property of such generating functions is that $y$ appears \emph{linearly} in the exponential.
(Indeed, as early as 1934, Meixner had essentially discovered the NEF-QVFs, albeit only recognizing their significance in terms of this distinctive property of their orthogonal polynomials \cite{Meixner-1934-OrthogonalPolynomialsGeneratingFunction,Lancaster-1975-DistributionsMeixnerClasses}.)

This linearity allows us to prove an addition formula, expanding the translation operator in orthogonal polynomials.
\begin{definition}[Translation polynomials]
    \label{def:translation-polynomials}
    Let $\tau_k(y; \mu) \in \RR[y]$ be defined by the generating function
    \begin{equation}
        \sum_{k \geq 0} \frac{z_{\mu}(t)^k}{k!}\tau_k(y; \mu) \colonequals \exp\big(y (\psi^\prime)^{-1}(t)\big).
    \end{equation}
    Also, define the normalized versions
    \begin{equation}
        \widehat{\tau}(y; \mu) \colonequals \frac{1}{V(\mu)^{k / 2}\sqrt{a_k(v_2)}}\tau_k(x; \mu).
    \end{equation}
\end{definition}
\begin{proposition}[Addition formula]
    For all $x, y \in \RR$ and $\mu \in \Omega$,
    \begin{equation}
        \label{eq:addition-formula}
        p_k(x + y; \mu) = \sum_{\ell = 0}^k \binom{k}{\ell}\tau_{k - \ell}(x; \mu) p_{\ell}(y; \mu).
    \end{equation}
\end{proposition}
\begin{proof}
    This follows from expanding the generating function \eqref{eq:op-gen-fn} at $x + y$ as a product of two exponential generating functions.
\end{proof}

Finally, we obtain the additively-spiked version of Corollary~\ref{cor:mismatched-mean} by taking expectations and using the orthogonality of the $p_k$.
\begin{proposition}[Additively-spiked expectation]
    \label{prop:additively-spiked-expectation}
    For all $k \in \NN$, $\mu \in \Omega$, and $x \in \RR$,
    \begin{equation}
        \Ex_{y \sim \widetilde{\rho}_{\mu}} \widehat{p}_k(x + y; \mu) = \widehat{\tau}_k(x; \mu).
    \end{equation}
\end{proposition}
\begin{proof}
    This follows from taking expectations on either side of \eqref{eq:addition-formula}, observing that the only non-zero term is for $\ell = 0$ by the orthogonality of the $p_{\ell}$, and noting that $p_{0}(y; \mu) = 1$.
\end{proof}

\subsection{Components of the likelihood ratio}
\label{sec:components}

Returning to the multivariate setting of our results, let $\QQ_n$ and $\PP_n$ be as in Definition~\ref{def:kin-spiked} of the kin-spiked NEF-QVF model.
Then, the likelihood ratio is
\begin{equation}
    \label{eq:lr-def}
    L_n(\by) \colonequals \frac{d\PP_n}{d\QQ_n}(\by) = \Ex_{\bx \sim \mathcal{P}_n}\left[ \prod_{i = 1}^{N}\frac{d\widetilde{\rho}_{x_i}}{d\widetilde{\rho}_{\mu_{n,i}}}(y_i)\right].
\end{equation}
An orthonormal system of polynomials for $\QQ_n$ is given by the product basis formed from the $\what{p}_k(y; \mu_{n, i})$ that we defined in Proposition~\ref{prop:nef-qvf-rodrigues}:
\begin{equation}
    \what{P}_{\bm k}(\bm y; \bm \mu_n) \colonequals \prod_{i = 1}^N \what{p}_{k_i}(y_i; \mu_{n, i})
\end{equation}
for $\bm k \in \NN^N$, where $\bm \mu_n \colonequals (\mu_{n, 1}, \ldots, \mu_{n, N(n)})$.

We show that the projection of $L_n$ onto any component $\what{P}_{\bm k}(\cdot; \bm \mu_n)$ admits the following convenient expression in terms of the $z$-score.
\begin{lemma}[Components under kin spiking]
    \label{lem:Ln-components-kin}
    In the kin-spiked NEF-QVF model, for all $\bm k \in \NN^N$,
    \begin{equation}
        \la L_n, \what{P}_{\bm k}(\cdot; \bm \mu_n) \ra = \sqrt{\frac{\prod_{i = 1}^N \what{a}_{k_i}(v_2)}{\prod_{i = 1}^N k_i!}}\Ex_{\bx \sim \sP_n} \left[\prod_{i = 1}^N z_{\mu_{n, i}}(x_i)^{k_i}\right].
    \end{equation}
\end{lemma}
\begin{proof}
    Performing a change of measure using the likelihood ratio and factorizing the inner product using independence of coordinates under $\QQ_n$, we find
    \begin{align*}
      \la L_n, \what{P}_{\bm k}(\cdot; \bm \mu_n) \ra
      &= \Ex_{\by \sim \QQ_n}\left[ L_n(\by) \what{P}_{\bm k}(\by; \bm \mu_n) \right] \\
      &= \Ex_{\by \sim \PP_n}\left[ \what{P}_{\bm k}(\by; \bm \mu_n) \right] \\
      &= \Ex_{\bx \sim \sP_n}\left[ \prod_{i = 1}^N \Ex_{y_i \sim \widetilde{\rho}_{x_i}}\left[ \what{p}_{k_i}(y_i; \mu_{n, i}) \right]\right] \\
        \intertext{and using Corollary~\ref{cor:mismatched-mean},}
      &= \sqrt{\frac{\prod_{i = 1}^N \what{a}_{k_i}(v_2)}{\prod_{i = 1}^N k_i!}} \Ex_{\bx \sim \sP_n}\left[ \prod_{i = 1}^N z_{\mu_{n, i}}(x_i)^{k_i} \right],
    \end{align*}
    completing the proof.
\end{proof}

Following the same argument for the additively-spiked model and using Proposition~\ref{prop:additively-spiked-expectation} instead of Corollary~\ref{cor:mismatched-mean} gives the following similar result.
\begin{lemma}[Components under additive spiking]
    \label{lem:Ln-components-additive}
    In the additively-spiked NEF-QVF model, for all $\bm k \in \NN^N$,
    \begin{equation}
        \la L_n, \what{P}_{\bm k}(\cdot; \bm \mu_n) \ra = \Ex_{\bx \sim \sP_n} \left[\prod_{i = 1}^N \widehat{\tau}_{k_i}(x_i; \mu_{n, i})\right].
    \end{equation}
\end{lemma}

\subsection{Full likelihood ratio norm}
\label{sec:full-norm}

First, we give an exact formula for the norm of the untruncated likelihood ratio in a kin-spiked NEF-QVF model.
\begin{theorem}
    \label{thm:full-lr}
    In the kin-spiked NEF-QVF model, for all $n \in \NN$,
    \begin{equation}
        \|L_n\|^2 = \Ex_{\bx^1, \bx^2 \sim \mathcal{P}_n}\left[ \prod_{i = 1}^N f(z_{\mu_{n, i}}(x_i^1)z_{\mu_{n, i}}(x^2_i); v_2) \right].
    \end{equation}
\end{theorem}
The key technical step is to recognize that the function $f(\cdot; v)$ from Definition~\ref{def:f} is in fact the exponential generating function of the $\what{a}_k(v)$, as follows.
\begin{proposition}
    \label{prop:f-gen-fun}
    For all $t \in \RR$ and $v \in \mathcal{V}$,
    \begin{equation}
        f(t; v) = \sum_{k = 0}^\infty \frac{\what{a}_k(v)}{k!} t^k.
    \end{equation}
\end{proposition}
\begin{proof}
    Differentiating the power series termwise and using the formula from Definition~\ref{def:a-consts} gives the differential equation
    \begin{equation}
        \frac{\partial}{\partial t} f(t; v) = f(t; v) + vt \frac{\partial}{\partial t} f(t; v),
    \end{equation}
    and the result follows upon solving the equation.
\end{proof}

\begin{proof}[Proof of Theorem~\ref{thm:full-lr}]
    We have by Lemma~\ref{lem:Ln-components-kin}
    \begin{align}
      \|L_n\|^2
      &= \sum_{\bm k \in \NN^N} \la L_n, \what{P}_{\bm k}(\cdot; \bm \mu_n) \ra^2 \nonumber \\
      &= \sum_{\bm k \in \NN^N}\frac{\prod_{i = 1}^N \what{a}_{k_i}(v_2)}{\prod_{i = 1}^N k_i!}\left(\Ex_{\bm x \sim \sP_n} \left[\prod_{i = 1}^N z_{\mu_{n, i}}(x_i)^{k_i}\right]\right)^2 \nonumber \\
      &= \Ex_{\bx^1, \bx^2 \sim \sP_n}\left[\sum_{\bk \in \NN^N}\prod_{i = 1}^N\left\{\frac{ \what{a}_{k_i}(v_2)}{k_i!} (z_{\mu_{n, i}}(x^1_i) z_{\mu_{n, i}}(x^2_i))^{k_i}\right\}\right] \nonumber \\
      &= \Ex_{\bx^1, \bx^2 \sim \sP_n} \left[\prod_{i = 1}^N \left\{\sum_{k = 0}^\infty \frac{\what{a}_k(v_2)}{k!}(z_{\mu_{n,i}}(x^1_i) z_{\mu_{n, i}}(x^2_i))^k\right\}\right],
    \end{align}
    and the result follows from Proposition~\ref{prop:f-gen-fun}.
\end{proof}

\subsection{Low-degree likelihood ratio norm: Proof of Theorem~\ref{thm:trunc-lr}}
\label{sec:trunc-norm}

For this result, we use two more ancillary facts about the constants $\what{a}_k(\cdot)$.

\begin{proposition}[Monotonicity]
    \label{prop:a-mon}
    For $k \in \NN$, $\what{a}_k(v)$ is non-negative and monotonically non-decreasing in $v$ over $v \in \mathcal{V}$.
\end{proposition}
\begin{proof}
    Recall from \eqref{eq:what-a} that, by definition,
    \begin{equation}
        \what{a}_k(v) = \prod_{j = 0}^{k - 1}(1 + vj).
    \end{equation}
    Thus clearly $\what{a}_k(v)$ is monotonically non-decreasing over $v \geq 0$, since each factor is monotonically non-decreasing.

    If $v \in \mathcal{V}$ with $v < 0$, then $v = -\frac{1}{m}$ for some $m \in \ZZ_{\geq 1}$.
    Thus for $k \geq m + 1$, $\what{a}_k(v) = 0$.
    So, in this case we may rewrite
    \begin{equation}
        \what{a}_k(v) = \One\{k \leq m\} \prod_{j = 0}^{\min\{k - 1, m - 1\}}(1 + vj).
    \end{equation}
    Now, each factor belongs to $[0, 1)$, and again each factor is monotonically non-decreasing with $v$, so the result follows.
\end{proof}

\begin{proposition}[Multiplicativity relations]
    \label{prop:a-mul}
    For all $\bm k \in \NN^N$,
    \[ \begin{array}{ll}
         \prod_{i = 1}^N \what{a}_{k_i}(v) \leq \what{a}_{\sum_{i = 1}^N k_i}(v) & \text{if } v > 0, \\
         \prod_{i = 1}^N \what{a}_{k_i}(v) = \what{a}_{\sum_{i = 1}^N k_i}(v) & \text{if } v = 0, \\
         \prod_{i = 1}^N \what{a}_{k_i}(v) \geq \what{a}_{\sum_{i = 1}^N k_i}(v) & \text{if } v < 0.
       \end{array}
   \]
\end{proposition}
\begin{proof}
    When $v = 0$, then $\what{a}_k(v) = 1$ for all $k$, so the result follows immediately.
    When $v > 0$, we have
    \begin{align}
      \prod_{i = 1}^N \what{a}_{k_i}(v)
      &= \prod_{i = 1}^N \prod_{j = 0}^{k_i - 1}(1 + vj) \nonumber \\
      &\leq \prod_{i = 1}^N \prod_{j = \sum_{a = 1}^{i - 1}k_a}^{\sum_{a = 1}^{i}k_a}(1 + vj) \nonumber \\
      &= \prod_{j = 1}^{\sum_{i = 1}^N k_i} (1 + vj) \nonumber \\
      &= \what{a}_{\sum_{i = 1}^N k_i}(v).
    \end{align}
    When $v < 0$, a symmetric argument together with the observations from Proposition~\ref{prop:a-mon} gives the result.
\end{proof}

\begin{proof}[Proof of Theorem~\ref{thm:trunc-lr}]
    Suppose first that $v_2 \geq 0$.
    We have by Lemma~\ref{lem:Ln-components-kin}
    \begin{align}
      \|L_n^{\leq D}\|^2
      &= \sum_{\substack{\bm k \in \NN^N \\ |\bm k| \leq D}} \la L_n, \what{P}_{\bm k}(\cdot; \bm \mu_n) \ra^2 \nonumber \\
      &= \sum_{\substack{\bm k \in \NN^N \\ |\bm k| \leq D}}\frac{\prod_{i = 1}^N \what{a}_{k_i}(v_2)}{\prod_{i = 1}^N k_i!}\left(\EE_{\bx \sim \sP_n} \left[\prod_{i = 1}^N z_{\mu_{n, i}}(x_i)^{k_i}\right]\right)^2 \nonumber \\
      &= \Ex_{\bx^1, \bx^2 \sim \sP_n}\left[\sum_{\substack{\bm k \in \NN^N \\ |\bm k| \leq D}}\frac{ \prod_{i = 1}^N \what{a}_{k_i}(v_2)}{\prod_{i = 1}^N k_i!} \prod_{i = 1}^N(z_{\mu_{n, i}}(x^1_i) z_{\mu_{n, i}}(x^2_i))^{k_i}\right], \nonumber
      \intertext{and using Proposition~\ref{prop:a-mul},}
      &\leq \Ex_{\bx^1, \bx^2 \sim \sP_n}\left[\sum_{\substack{\bm k \in \NN^N \\ |\bm k| \leq D}}\frac{ \what{a}_{|\bm k|}(v_2)}{\prod_{i = 1}^N k_i!} \prod_{i = 1}^N(z_{\mu_{n, i}}(x^1_i) z_{\mu_{n, i}}(x^2_i))^{k_i}\right] \nonumber \\
      &= \Ex_{\bx^1, \bx^2 \sim \sP_n}\left[\sum_{d = 0}^D\frac{\what{a}_d(v_2)}{d!}\sum_{\substack{\bm k \in \NN^N \\ |\bm k| = d}}
      \binom{d}{k_1 \cdots k_N} \prod_{i = 1}^N(z_{\mu_{n, i}}(x^1_i) z_{\mu_{n, i}}(x^2_i))^{k_i}\right] \nonumber \\
      &= \Ex_{\bx^1, \bx^2 \sim \sP_n}\left[\sum_{d = 0}^D\frac{\what{a}_d(v_2)}{d!}
        \left(\sum_{i = 1}^Nz_{\mu_{n, i}}(x^1_i) z_{\mu_{n, i}}(x^2_i)\right)^d\right],
    \end{align}
    giving the upper bound from \eqref{eq:trunc-lr-pos} for $v_2 > 0$.
    When $v_2 = 0$, then equality holds above, so we obtain equality in \eqref{eq:trunc-lr-pos}.
    Also, when $v_2 < 0$, then the above argument holds with the inequality reversed, giving the lower bound of \eqref{eq:trunc-lr-neg}.

    Finally, for the upper bound of \eqref{eq:trunc-lr-neg}, note that when $v_2 < 0$, we may bound $\|L_n^{\leq D}\|^2$ using Proposition~\ref{prop:a-mon} and the result for $v_2 = 0$ by
    \begin{align}
      \|L_n^{\leq D}\|^2
      &= \sum_{\substack{\bm k \in \NN^N \\ |\bm k| \leq D}}\frac{\prod_{i = 1}^N \what{a}_{k_i}(v_2)}{\prod_{i = 1}^N k_i!}\left(\EE_{\bx \sim \sP_n} \left[\prod_{i = 1}^N z_{\mu_{n, i}}(x_i)^{k_i}\right]\right)^2 \nonumber \\
      &\leq \sum_{\substack{\bm k \in \NN^N \\ |\bm k| \leq D}}\frac{\prod_{i = 1}^N \what{a}_{k_i}(0)}{\prod_{i = 1}^N k_i!}\left(\EE_{\bx \sim \sP_n} \left[\prod_{i = 1}^N z_{\mu_{n, i}}(x_i)^{k_i}\right]\right)^2 \nonumber \\
      &= \Ex_{\bx^1, \bx^2} f^{\leq D}(r_n; 0),
  \end{align}
  giving the result.
\end{proof}

\section{Applications}

\subsection{Example: Kesten-Stigum on the back of an envelope}
\label{sec:kesten-stigum}

Let us show how to use Theorem~\ref{thm:trunc-lr} to predict a computational threshold in the symmetric stochastic block model with two communities (see, e.g., \cite{Abbe-2017-SBMReview,Moore-2017-SBMReview} for surveys of this model).
In this model, $N(n) = \binom{n}{2}$, and indexing in $N$ is identified with pairs $\{i, j\} \in \binom{[n]}{2}$, which represent edges in a graph.
There are two external parameters, $a, b > 0$.
The model belongs to the Bernoulli NEF-QVF, with $v_2 = -1$, and the means in the null and planted models are as follows:
\begin{itemize}
\item Under $\QQ_n$, $\mu_{n, \{i, j\}} = \frac{a + b}{2n}$, with variances $V(\mu_{n, \{i, j\}}) = \mu_{n, \{i, j\}}(1 - \mu_{n, \{i, j\}}) \approx \mu_{n, \{i, j\}}$.
\item Under $\PP_n$, we generate $\bm\sigma \in \{\pm 1\}^N$ either i.i.d.\ uniformly or conditioned to have equal numbers of plus and minus coordinates (this does not make a significant difference), and set $x_{\{i, j\}} = \frac{a}{n}$ if $\sigma_i = \sigma_j$ and $x_{\{i, j\}} = \frac{b}{n}$ if $\sigma_i \neq \sigma_j$.
    We may summarize this as
    \begin{equation}
        x_{\{i, j\}} = \frac{a + b}{2n} + \frac{a - b}{2n}\sigma_i \sigma_j.
    \end{equation}
\end{itemize}
We first compute the $z$-scores:
\begin{equation}
    z_{\mu_{n, \{i, j\}}}(x_{\{i, j\}}) = \frac{\left(\frac{a + b}{2n} + \frac{a - b}{2n}\sigma_i \sigma_j\right) - \frac{a + b}{2n}}{\sqrt{\frac{a + b}{2n}}} = \frac{1}{\sqrt{2n}} \cdot \frac{a - b}{\sqrt{a + b}}\cdot \sigma_i \sigma_j.
\end{equation}
As a shorthand, let us write $\bz(\bx)$ for the coordinatewise application of this function.
We then compute the inner product of the $z$-scores of two independent copies of $\bx$, which may be expressed in terms of two copies of $\bm\sigma$:
\begin{equation}
    r_n \colonequals \langle \bz(\bx^1), \bz(\bx^2) \rangle = \frac{1}{2n} \frac{(a - b)^2}{a + b}\sum_{1 \leq i < j \leq n}\sigma_i^1 \sigma_i^2 \sigma_j^1\sigma_j^2 = \frac{(a - b)^2}{4(a + b)} \cdot \frac{ \la \bm\sigma^1, \bm\sigma^2 \ra^2 - n}{n}.
\end{equation}
By our heuristic based on Theorem~\ref{thm:trunc-lr}, we then expect
\begin{equation}
    \|L_n^{\leq D}\|^2 \approx \Ex_{\bm\sigma^1, \bm\sigma^2}\left[\exp^{\leq D}\left(\frac{(a - b)^2}{4(a + b)} \cdot \left(\frac{\la \bm\sigma^1, \bm\sigma^2 \ra^2}{n} - 1\right)\right)\right].
\end{equation}
(Alternatively, if we work in the Poisson NEF instead of the Bernoulli NEF, then this will be an exact equality by Theorem~\ref{thm:trunc-lr}.)

We assume heuristically that $\la \bm\sigma^1, \bm\sigma^2 \ra / \sqrt{n}$ is distributed approximately as $\sN(0, 1)$, and that $D(n)$ grows slowly enough that only after this convergence does the convergence $\exp^{\leq D}(\cdot) \to \exp(\cdot)$ occur.
Thus, we expect
\begin{equation}
    \limsup_{n \to \infty} \|L_n^{\leq D(n)}\|^2 \lesssim \Ex_{g \sim \sN(0, 1)}\left[ \exp\left(\frac{(a - b)^2}{4(a + b)}(g^2 - 1)\right)\right],
\end{equation}
where the right-hand side evaluates the moment-generating function of a $\chi^2$ random variable, which is finite if and only if $\frac{(a - b)^2}{4(a + b)} < \frac{1}{2}$, or if and only if $(a - b)^2 < 2(a + b)$, which is the Kesten-Stigum threshold.

\subsection{Channel monotonicity: Proof of Theorem~\ref{thm:comparison}}

This result is a simple consequence of the arguments we have made already to prove Theorem~\ref{thm:trunc-lr}.
\begin{proof}[Proof of Theorem~\ref{thm:comparison}]
    We have by Lemma~\ref{lem:Ln-components-kin}
    \begin{align}
      \|(L_n^{(i)})^{\leq D}\|^2
      &= \sum_{\substack{\bm k \in \NN^N \\ |\bm k| \leq D}} \la L_n^{(i)}, \what{P}_{\bm k}(\cdot; \bm \mu_n) \ra^2 \nonumber \\
      &= \sum_{\substack{\bm k \in \NN^N \\ |\bm k| \leq D}}\frac{\prod_{j = 1}^N \what{a}_{k_j}(v_2^{(i)})}{\prod_{j = 1}^N k_j!}\left(\EE_{\bx \sim \sP_n} \left[\prod_{j = 1}^N z_{\mu_{n, j}}(x_j)^{k_j}\right]\right)^2.
    \end{align}
    In each term on the right-hand side, the only factor that depends on $i$ is $\prod_{j = 1}^N\what{a}_{k_j}(v_2^{(i)})$, so the result follows from the monotonicity described by Proposition~\ref{prop:a-mon}.
    (Indeed, this shows slightly more, that the monotonicity holds even for the norm of the projection of $L_n^{(i)}$ onto the orthogonal polynomial of any given index $\bk$.)
\end{proof}

\subsection{Hyperbolic secant spiked matrix model: Proof of Theorem~\ref{thm:spiked-mx}}

To prove this result, we will analyze the translation polynomials $\tau_k$, from Definition~\ref{def:translation-polynomials}, for the NEF generated by $\rho^{\sech}$.
First, note that the mean and variance of $\rho^{\sech}$ are $\mu = 0$ and $V(0) = 1$, and more generally the variance function in the generated NEF is $V(\mu) = \mu^2 + 1$ (per Table~\ref{tab:nef-qvf}), where in particular the quadratic coefficient is $v_2 = 1$.
Thus the associated normalizing constants are $a_k(v_2) = (k!)^2$ and $\widehat{a}_k(v_2) = k!$.

Recall that the translation polynomials admit a generating function expressed in terms of the cumulant generating function of $\rho^{\sech}$.
We therefore compute
\begin{align}
  \psi(\theta) &\colonequals \Ex_{y \sim \rho^{\sech}} \exp(\theta y) = \frac{1}{2}\int_{-\infty}^{\infty} \sech\left(\frac{\pi y}{2}\right) \exp(\theta y) \, dy = -\log(\cos \theta), \\
  \psi^{\prime}(\theta) &= \tan(\theta),
\end{align}
whereby the translation polynomials for $\mu = 0$ (the mean of $\rho^{\sech}$) have the generating function
\begin{equation}
    \sum_{k \geq 0} \frac{t^k}{k!} \tau_k(y; 0) = \sum_{k \geq 0} t^k\widehat{\tau}_k(y; 0) = \exp\left(y \tan^{-1}(t)\right). \label{eq:sech-translation-gf}
\end{equation}

Before proceeding, we also establish some preliminary bounds on the coefficients and values of these polynomials.
We denote by $[x^\ell](p(x))$ the coefficient of $x^{\ell}$ in a polynomial or formal power series $p(x)$.
\begin{proposition}
    \label{prop:tau-coeff-bound}
    For all $k \geq 1$ and $\ell \geq 0$,
    \begin{equation}
        |[x^{\ell}](\what{\tau}_k(x))| \leq \One\{k \equiv \ell \ppmod{2}, \ell > 0\} \frac{(2\log(ek))^{\ell - 1}}{k \, \ell!}.
    \end{equation}
\end{proposition}
\begin{proof}
    Expanding the generating function, we have
    \begin{equation}
        [x^{\ell}](\what{\tau}_k(x)) = [t^kx^{\ell}](\exp(x\tan^{-1}(t))) = \frac{1}{\ell!} [t^k]((\tan^{-1}(t))^{\ell}).
    \end{equation}
    If $k \geq 1$ and $\ell = 0$, then this is zero.
    Since the coefficients in the Taylor series of $\tanh^{-1}(t)$ are $[t^k](\tanh^{-1}(t)) = \One\{k \equiv 1 \ppmod{2}\} (-1)^{(k - 1) / 2} / k$, we may bound
    \begin{equation}
        |[x^{\ell}](\what{\tau}_k(x))| \leq \One\{k \equiv \ell \ppmod{2}, \ell > 0\} \frac{1}{\ell!} \underbrace{\sum_{\substack{a_1, \dots, a_{\ell} \geq 1 \\ a_1 + \cdots + a_{\ell} = k}} \frac{1}{\prod_{i = 1}^{\ell}a_i}}_{c(k, \ell)}.
    \end{equation}

    We now show that $c(k, \ell) \leq (2\log(ek))^{\ell - 1} / k$ by induction on $\ell$.
    Since $c(k, 1) = 1 / k$, the base case holds.
    We note the bound on harmonic numbers
    \begin{equation}
        \label{eq:harm-bound}
        \sum_{a = 1}^k \frac{1}{a} \leq \log(ek) \text{ for all } k \geq 1.
    \end{equation}
    Supposing the result holds for $c(k, \ell - 1)$, we expand $c(k, \ell)$ according to the value that $a_{\ell}$ takes:
    \begin{align}
      c(k, \ell)
      &\leq \sum_{a = 1}^{k - 1} \frac{1}{a}c(k - a, \ell - 1) \nonumber \\
      &\leq (2\log (ek))^{\ell - 2} \sum_{a = 1}^{k - 1} \frac{1}{a} \cdot \frac{1}{k - a} \tag{inductive hypothesis} \\
      &\leq \frac{(2\log (ek))^{\ell - 2}}{k}\sum_{a = 1}^{k - 1} \left(\frac{1}{a} + \frac{1}{k - a}\right) \nonumber \\
      &\leq \frac{(2\log (ek))^{\ell - 2}}{k} \cdot 2\log(ek), \tag{by \eqref{eq:harm-bound}}
    \end{align}
    completing the argument.
\end{proof}

This yields the following pointwise bound.
As we will ultimately be evaluating this on quantities of order $O(n^{-1/2})$, what is most important to us is the precision for very small arguments.
\begin{corollary}
    \label{cor:tau-value-bound}
    For all $k \geq 1$ and $x > 0$,
    \begin{equation}
        |\what{\tau}_k(x)| \leq \left\{\begin{array}{ll} x \cdot \frac{1}{k} \cdot (ek)^{2x} & \text{if } k \text{ odd}, \\ x^2 \cdot \frac{2\log(ek)}{k} \cdot (ek)^{2x} & \text{if } k \text{ even.}\end{array}\right.
    \end{equation}
\end{corollary}
\begin{proof}
    Write $\ell_0 = 1$ if $k$ is odd and $\ell_0 = 2$ if $k$ is even.
    We bound by Proposition~\ref{prop:tau-coeff-bound},
    \begin{align}
      |\what{\tau}_k(x)|
      &\leq \frac{1}{k}\sum_{\ell = \ell_0}^k \frac{(2\log(ek))^{\ell - 1}}{\ell!} x^{\ell} \nonumber \\
      &\leq \frac{x^{\ell_0} (2\log(ek))^{\ell_0 - 1}}{k}\sum_{\ell = \ell_0}^k \frac{(2\log(ek) x)^{\ell - \ell_0}}{(\ell - \ell_0)!} \nonumber \\
      &\leq \frac{x^{\ell_0} (2\log(ek))^{\ell_0 - 1}}{k} \exp((2\log(ek) x),
    \end{align}
    and the result follows upon rearranging.
\end{proof}

We now proceed with the proof of the main result of this section.
\begin{proof}[Proof of Theorem~\ref{thm:spiked-mx}]
    First, applying Lemma~\ref{lem:Ln-components-additive} to the hyperbolic secant spiked matrix model, the coefficients of the likelihood ratio are given by, for any $\bk \in \NN^{\binom{[n]}{2}}$,
    \begin{equation}
    \langle L_n(\bY), \widehat{P}_{\bk}\rangle = \Ex_{\bX \sim \sP_n} \left[ \prod_{1 \leq i < j \leq n} \what{\tau}_{k_{\{i, j\}}}(X_{\{i, j\}}) \right] = \Ex_{\bx \sim \Unif(\{\pm 1\}^n)} \left[ \prod_{1 \leq i < j \leq n} \what{\tau}_{k_{\{i, j\}}}\left(\frac{\lambda}{\sqrt{n}} x_ix_j\right) \right].
    \end{equation}
    We proceed in several steps.

    \vspace{1em}

    \emph{Step 1: Simplifying Rademacher-valued prior.}
    Our specific choice of $\bx \in \{ \pm 1\}^n$ allows an interesting further simplification: thanks to this choice, we can decouple the dependence of the components of $L_n$ on $\lambda$ from the dependence on $\bx$.
    Note that, by the generating function identity~\eqref{eq:sech-translation-gf}, for all $k \geq 0$ we have that $\tau_k(x)$ contains only monomials of the same parity as $k$.
    Therefore, for all $\bk \in \NN^{\binom{[n]}{2}}$, we have
    \begin{equation}
        \langle L_n, \what{P}_{\bk} \rangle = \prod_{i < j} \what{\tau}_{k_{ij}}\left(\frac{\lambda}{\sqrt{n}}\right) \cdot \Ex_{\bx} \left[ \prod_{i < j} (x_ix_j)^{k_{ij}} \right].
    \end{equation}
    Here and in the remainder of the proof, we write $k_{ij} = k_{\{i, j\}}$ and $i < j$ for $1 \leq i < j \leq n$ to lighten the notation.
    Let us also write $|\bk|_{\infty} \colonequals \max_{i < j} k_{ij}$.

    \vspace{1em}

    \emph{Step 2: Bounds on translation polynomials.}
    We note that, since the second factor above is either 0 or 1, we may further bound
    \begin{align}
      |\langle L_n, \what{P}_{\bk} \rangle|
      &\leq \left|\prod_{i < j} \what{\tau}_{k_{ij}}\left(\frac{\lambda}{\sqrt{n}}\right)\right| \cdot \Ex_{\bx} \left[ \prod_{i < j} (x_ix_j)^{k_{ij}} \right] \nonumber
        \intertext{When $|\bk|_{\infty} \leq D$, then, by Corollary~\ref{cor:tau-value-bound}, we may continue}
      &\leq \prod_{\substack{i < j \\ k_{ij} > 0}}\frac{(eD)^{\frac{\lambda}{\sqrt{n}}}}{k_{ij}}(2\log(ek_{ij}))^{\One\{ k_{ij} \text{ even}\}}\left(\frac{\lambda}{\sqrt{n}}\right)^{1 + \One\{k_{ij} \text{ even}\}}\Ex_{\bx}\left[\prod_{i < j} (x_ix_j)^{k_{ij}}\right].
    \end{align}

    \vspace{1em}

    \emph{Step 3: The ``replica'' manipulation.}
    Squaring and rewriting this as an expectation over two independent $\bx^1, \bx^2 \sim \Unif(\{\pm 1\}^n)$, we find
    \begin{align}
      |\langle L_n, \what{P}_{\bk} \rangle|^2 &\leq \Ex_{\bx^1, \bx^2} \prod_{\substack{i < j \\ k_{ij} > 0}}\frac{(eD)^{2\frac{\lambda}{\sqrt{n}}}}{k_{ij}^2}(2\log(ek_{ij}))^{2\, \One\{ k_{ij} \text{ even}\}}\left(\frac{\lambda}{\sqrt{n}}\right)^{2(1 + \One\{k_{ij} \text{ even}\})} (x_i^1x_i^2x_j^1x_j^2)^{k_{ij}}.
    \end{align}
    Summing over $|\bk|_{\infty} \leq D$, we then find
    \begin{align}
      &\sum_{\substack{\bk \in \NN^N \\ |\bk|_{\infty} \leq D}}|\langle L_n, \what{P}_{\bk} \rangle|^2 \nonumber \\
      &\leq \Ex_{\bx^1, \bx^2}\prod_{i < j}\left(1 + (eD)^{2\frac{\lambda}{\sqrt{n}}} \frac{\lambda^2}{n}\sum_{\substack{k = 1 \\ k \text{ odd}}}^D \frac{1}{k^2}(x_i^1x_i^2x_j^1x_j^2)^{k} + (eD)^{2\frac{\lambda}{\sqrt{n}}} \frac{\lambda^4}{n^2}\sum_{\substack{k = 1 \\ k \text{ even}}}^D \frac{4\log(ek)^2}{k^2}(x_i^1x_i^2x_j^1x_j^2)^{k} \right) \nonumber
      \intertext{and, using that the $x_i^a$ are Rademacher-valued,}
      &= \Ex_{\bx^1, \bx^2}\prod_{1 \leq i < j \leq n}\left(1 + (eD)^{2\frac{\lambda}{\sqrt{n}}} \frac{\lambda^2}{n}x_i^1x_i^2x_j^1x_j^2\sum_{\substack{k = 1 \\ k \text{ odd}}}^D \frac{1}{k^2} + (eD)^{2\frac{\lambda}{\sqrt{n}}} \frac{\lambda^4}{n^2}\sum_{\substack{k = 1 \\ k \text{ even}}}^D \frac{4\log(ek)^2}{k^2} \right) \nonumber
      \intertext{Here, using that $\sum_{\ell \geq 0} \frac{1}{(2\ell + 1)^2} = \frac{\pi^2}{8} = \lambda_*^{-2}$ and $\sum_{\ell \geq D / 2} \frac{1}{(2\ell + 1)^2} \geq \int_{D / 2}^{\infty} \frac{dx}{(2x + 1)^2} = \frac{1}{2D + 2} \geq \frac{1}{3D}$, we may write}
      &= \Ex_{\bx^1, \bx^2}\prod_{1 \leq i < j \leq n}\left(1 + (eD)^{2\frac{\lambda}{\sqrt{n}}} \frac{\lambda^2}{n}x_i^1x_i^2x_j^1x_j^2 \left(\frac{1}{\lambda_*^2} - \frac{1}{3D}\right) + O\left(\frac{1}{n^2} \right)\right) \nonumber \\
      &\leq \Ex_{\bx^1, \bx^2}\exp\left((eD)^{2\frac{\lambda}{\sqrt{n}}} \frac{\lambda^2}{n} \left(\frac{1}{\lambda_*^2} - \frac{1}{3D}\right)\sum_{1 \leq i < j \leq n} x_i^1x_i^2x_j^1x_j^2  + O(1)\right) \nonumber \\
      &= \Ex_{\bx^1, \bx^2}\exp\left((eD)^{2\frac{\lambda}{\sqrt{n}}} \frac{\lambda^2}{2} \left(\frac{1}{\lambda_*^2} - \frac{1}{3D}\right)\frac{\langle \bx^1, \bx^2 \rangle^2}{n} + O(1)\right), \nonumber
        \intertext{where we absorb the diagonal terms from $\langle \bx^1, \bx^2 \rangle^2 / n$ into the $O(1)$ term.
        Finally, by our assumption we have $\lambda < \lambda_* + \frac{1}{20D}$.
        Therefore, $\lambda^2 < \lambda_*^2 + \frac{41}{400D}$.
        So, $\lambda^2(\lambda_*^{-2} - \frac{1}{3D}) < 1 - \frac{1}{6D} + \frac{41}{400D} < 1 - \frac{1}{20D}$, and thus, for sufficiently large $n$, we will have}
        &\leq \Ex_{\bx^1, \bx^2}\exp\left( \frac{1}{2}\left(1 - \frac{1}{20D}\right)\frac{\langle \bx^1, \bx^2 \rangle^2}{n} + O(1)\right).
    \end{align}

    This is precisely the quantity arising in the analysis of computationally-unbounded strong detection for the Gaussian Rademacher-spiked matrix model in \cite{PWBM-2018-PCAI} (invoking Le Cam's second moment method as mentioned above), where it is shown that this quantity is bounded as $n \to \infty$, since the factor multiplying $\langle \bx^1, \bx^2 \rangle^2 / n$ is strictly smaller than $\frac{1}{2}$.
    (This makes formal the heuristic argument that $\langle \bx^1, \bx^2 \rangle / \sqrt{n}$ converges to a standard Gaussian random variable, whereby the above is asymptotically an evaluation of the moment generating function of a $\chi^2$ random variable, which we also alluded to in Section~\ref{sec:kesten-stigum}.)
    Thus we find
    \begin{equation}
        \limsup_{n \to \infty}\sum_{\substack{\bk \in \NN^N \\ |\bk|_{\infty} \leq D}}|\langle L_n, \what{P}_{\bk} \rangle|^2 < +\infty,
    \end{equation}
    as claimed.
\end{proof}

\begin{remark}[A general ``Rademacher trick'']
    Step 1 in the proof above, where we take advantage of the Rademacher prior to decouple the dependence of the likelihood ratio's components on the signal-to-noise ratio $\lambda$ from that on the actual spike vector $\bx$, should apply in much greater generality.
    Indeed, we expect a similar property to hold in any additive model where (1) the spike distribution $\bx \sim \sP_n$ has the property that $|x_i| = \lambda(n)$ for some constant $\lambda(n)$ for all $i \in [N(n)]$, and (2) the noise distribution is symmetric, whereby the polynomials playing the role of $\widehat{\tau}_k$ will be even polynomials for even $k$ and odd polynomials for odd $k$.
    Thus a similar analysis is likely possible in a wide range of models with ``flat'' signals $\bx$, reducing the low-degree analysis to analytic questions about $\what{\tau}_k$.
\end{remark}

\begin{remark}
    The original argument of \cite{PWBM-2018-PCAI} derives the critical value $\lambda^*$ for $\Wig(\rho^{\sech}, \lambda)$ in terms of the Fisher information in the family of translates of the distribution $\rho^{\sech}$, while our calculation, if we consider $D = D(n)$ growing slowly, obtains the same value using orthogonal polynomials.
    It appears that the connection between these derivations lies in the summation identity $\sum_{\ell \geq 0} \frac{1}{(2\ell + 1)^2} = \frac{\pi^2}{8}$.
    We suspect that there are similar identities associated to these two approaches to calculating the critical signal-to-noise ratio in $\Wig(\rho, \lambda)$ for other well-behaved measures $\rho$.
    It would be interesting to understand what class of summation identities arises in this way, and whether equating these two derivations can give novel proofs of such identities.
\end{remark}

\subsection{Mixed spiked matrix model: Proof of Theorem~\ref{thm:mixed-spiked-mx}}

This result follows almost immediately from Theorem~\ref{thm:spiked-mx} upon expanding the definitions.

\begin{proof}[Proof of Theorem~\ref{thm:mixed-spiked-mx}]
    Recall that $\rho^{\heavy}$ has density proportional to $(1 + x^2)^{-\alpha / 2}$, and we write $\Wig(\rho^{\sech}, \lambda_* \cdot \lambda) \equalscolon ((\QQ^{(1)}_n, \PP^{(1)}_n))_{n = 1}^{\infty}$ and $\Wig(\rho^{\heavy}, \lambda) \equalscolon ((\QQ^{(2)}_n, \PP^{(2)}_n))_{n = 1}^{\infty}$.
    Note that, since $\EE_{x \sim \rho^{\heavy}} |x|^{\beta} = +\infty$ for all $\beta \geq \alpha - 1$, all polynomials in $L^2(\QQ^{(2)}_n)$ have entrywise degree at most $\alpha / 2$ in every coordinate.

    Since $\QQ_n$ is a mixture of $\QQ^{(1)}_n$ and $\QQ^{(2)}_n$ with weight $\frac{1}{2}$ for each, we have $L^2(\QQ_n) = L^2(\QQ^{(1)}_n) \cap L^2(\QQ^{(2)}_n)$ and $\EE_{\bY \sim \QQ_n} f_n(\bY)^2 = \frac{1}{2}\EE_{\bY \sim \QQ_n^{(1)}} f_n(\bY)^2 + \frac{1}{2}\EE_{\bY \sim \QQ_n^{(2)}} f_n(\bY)^2$.
    Therefore,
    \begin{align}
      &\left\{\begin{array}{ll} \text{maximize} & \EE_{\bY \sim \PP_n} f_n(\bY) \\[0.5em] \text{subject to} & f_n \in \RR[\bY] \cap L^2(\QQ_n) \\[0.5em] & \EE_{\bY \sim \QQ_n} f_n(\bY)^2 = 1 \end{array}\right\} \nonumber \\[0.5em]
      &\hspace{2cm}\leq \left\{\begin{array}{ll} \text{maximize} & \EE_{\bY \sim \PP_n^{(1)}} f_n(\bY) \\[0.5em] \text{subject to} & f_n \in \RR[\bY] \setminus \{0\} \\[0.5em] & \deg_{\{i, j\}}(f_n) \leq \alpha / 2 \text{ for all } \{i, j\} \in \binom{[n]}{2} \\[0.5em] & \EE_{\bY \sim \QQ_n^{(1)}} f_n(\bY)^2 \leq 2 \end{array}\right\} \nonumber \\[0.5em]
      &\hspace{2cm}\leq \left\{\begin{array}{ll} \text{maximize} & \EE_{\bY \sim \PP_n^{(1)}} f_n(\bY) \\[0.5em] \text{subject to} & f_n \in \RR[\bY] \\[0.5em] & \deg_{\{i, j\}}(f_n) \leq \alpha / 2 \text{ for all } \{i, j\} \in \binom{[n]}{2} \\[0.5em] & \EE_{\bY \sim \QQ_n^{(1)}} f_n(\bY)^2 = 1 \end{array}\right\} \cdot \sqrt{2},
    \end{align}
    where the final expression is precisely that controlled by Theorem~\ref{thm:spiked-mx}, and the result follows upon taking $D = \alpha / 2$ in that result.
\end{proof}

\section*{Acknowledgements}
\addcontentsline{toc}{section}{Acknowledgements}

I thank Alex Wein for comments on an early version of the manuscript, and Afonso Bandeira for helpful discussions.

\bibliographystyle{alpha}
\addcontentsline{toc}{section}{References}
\bibliography{main}

\end{document}